\documentclass[12pt]{amsart}
\usepackage{amssymb}
\usepackage{graphicx}
\newtheorem{theorem}{Theorem}[section]
\newtheorem{lemma}[theorem]{Lemma}
\newtheorem{corollary}[theorem]{Corollary}
\newtheorem{proposition}[theorem]{Proposition}
\newtheorem{remark}[theorem]{Remark}
\newtheorem{definition}[theorem]{Definition}
\newtheorem{example}[theorem]{Example}
\newtheorem*{question*}{Question}
\newcommand{\dsubseteq}{\mathrel{\rotatebox[origin=c]{45}{$\subseteq$}}}
\newcommand{\bsubseteq}{\mathrel{\rotatebox[origin=c]{-45}{$\subseteq$}}}
\begin{document}
\title[Structure and a duality of binary operations on monoids and groups]{Structure and a duality of binary operations on monoids and groups}
\author{Masayoshi Kaneda}
\address{Department of Mathematics and Natural Sciences, College of Arts and Sciences, American University of Kuwait, P.O. Box 3323, Safat 13034 Kuwait}
\email{mkaneda@uci.edu}
\date{\today}
\begin{abstract}In this paper we introduce novel views of monoids and groups. More specifically, for a given set $S$, let $S^{S\times S}$ be the set of binary operations on $S$. We equip $S^{S\times S}$ with canonical binary operations induced by the elements of $S$. Let $S^{S\times S}_{mn}$ (respectively, $S^{S\times S}_{gr}$) be the set of binary operations that make $S$ monoids (respectively, groups). Then we have the following ``duality'': for each $z\in S^{S\times S}_{mn}$ a certain subset of $S^{S\times S}$, denoted by $S^*_z$, is a monoid with a canonical binary operation and is isomorphic to $(S,z)$. If $z\in S^{S\times S}_{gr}$, then $S^{S\times S}_{gr}$ can be partitioned into copies of $S^*_z$. We also give a new characterization of group binary operations which distinguishes them from the other binary operations. These results give us new insights into monoids and groups, and will provide new tools and directions in studying these objects.
\end{abstract}
\maketitle
\section{Introduction and Preliminaries.}\label{section:intro}
Let us recall very basic definitions. A \emph{magma} is a set equipped with a binary operation. A \emph{semigroup} is a magma in which its binary operation is associative. A \emph{monoid} is a semigroup with a two-sided identity. A \emph{group} is a monoid in which every element has a two-sided inverse. A \emph{homomorphism} is a function between magmas preserving their binary operations. An \emph{isomorphism} is a homomorphism which is bijective (i.e., one-to-one and onto). An \emph{automorphism} is an isomorphism from a magma onto itself.

For a given set $S$, let us denote by $S^{S\times S}$ the set of all binary operations on $S$; that is, the set of all functions from $S\times S$ to $S$. Each element of $S^{S\times S}$ makes $S$ a distinct magma if no isomorphic identification is made. In general the set $S^{S\times S}$ is huge compared with $S$ if $S$ is nontrivial. Indeed, if $S$ is a nonempty finite set, then $|S^{S\times S}|=|S|^{|S|^2}$. Thus if $|S|=2$, then $|S^{S\times S}|=16$; if $|S|=3$, then $|S^{S\times S}|=19683$; if $|S|=4$, then $|S^{S\times S}|=4294967296$; so on. In this paper, however, we mostly deal with relatively small subsets of $S^{S\times S}$ with certain properties which we shall define in what follows.

Let us denote a magma $S$ with binary operation $z\in S^{S\times S}$ by the ordered pair $(S,z)$. We denote by $azb\in S$ the image of $(a,b)\in S\times S$ by $z\in S^{S\times S}$.

Now we define the following subclasses of $S^{S\times S}$.

$S^{S\times S}_{sg}$ is the set of elements $z\in S^{S\times S}$ for which $(S,z)$ is a semigroup; that is, the binary operation $z$ is associative:$$S^{S\times S}_{sg}:=\{z\in S^{S\times S}\,|\,(azb)zc=az(bzc),\forall a,b,c\in S\}.$$

$S^{S\times S}_{mn}$ is the set of elements $z\in S^{S\times S}_{sg}$ for which $(S,z)$ is a monoid; that is, $(S,z)$ is a semigroup with a two-sided identity (necessarily unique):$$S^{S\times S}_{mn}:=\{z\in S^{S\times S}_{sg}\,|\,\exists e\in S\text{ such that }eza=a=aze,\forall a\in S\}.$$

$S^{S\times S}_{gr}$ is the set of elements $z\in S^{S\times S}_{mn}$ for which $(S,z)$ is a group; that is, $(S,z)$ is a monoid in which every element has a two-sided inverse (necessarily unique):$$S^{S\times S}_{gr}:=\{z\in S^{S\times S}_{mn}\,|\,\forall a\in S, \exists b\in S\text{ such that }bza=azb=e\},$$where $e$ is the identity of $(S,z)$.

Besides, we define the following.
\begin{definition}\label{de:nondegenerate}{\em We say that an element $z\in S^{S\times S}$ is \textbf{nondegenerate} if $z$ is onto; that is, for every $a\in S$ there exist $b,c\in S$ such that $a=bzc$. We denote by $S^{S\times S}_{nd}$ the set of all nondegenerate elements in $S^{S\times S}$.}
\end{definition}Clearly,
\begin{alignat*}{5}
    && && S^{S\times S}_{nd} && && \\
    && \dsubseteq && && \bsubseteq && \\
    S^{S\times S}_{gr}\subseteq S^{S\times S}_{mn}\subseteq S^{S\times S}_{nd}\cap S^{S\times S}_{sg} && && && && S^{S\times S}\\
    && \bsubseteq && && \dsubseteq && \\
    && && S^{S\times S}_{sg} && &&
\end{alignat*}and these sets are nonempty as long as $S\neq\emptyset$. It is not hard to see that if $|S|\ge2$, then each inclusion is proper and neither $S^{S\times S}_{nd}$ nor $S^{S\times S}_{sg}$ is included in the other. However, elements of $S^{S\times S}_{nd}$ and $S^{S\times S}_{sg}$ are closely related each other. Indeed, any element of $S^{S\times S}$ that is compatible with an element of $S^{S\times S}_{nd}$ must be an element of $S^{S\times S}_{sg}$; that is, such an element must be associative. See Proposition~\ref{pr:dual}~(1) and Question after the proposition.

When we try to equip the set $S^{S\times S}$ with a binary operation naturally induced by $a\in S$, we encounter a situation in which we must consider associativity involving more than one binary operation on $S$ as seen in what follows.

Each element $a\in S$ induces two canonical binary operations $\hat{\,}$ and $\check{\,}$ on $S^{S\times S}$:
$$b(z_1\hat{a}z_2)c:=(bz_1a)z_2c,\quad\forall b,c\in S,\quad\forall z_1,z_2\in S^{S\times S}.$$
$$b(z_1\check{a}z_2)c:=bz_1(az_2c),\quad\forall b,c\in S,\quad\forall z_1,z_2\in S^{S\times S}.$$
\begin{definition}{\em We say that $z_1\in S^{S\times S}$ and $z_2\in S^{S\times S}$ are \textbf{compatible} if $z_1\hat{a}z_2=z_1\check{a}z_2$ and $z_2\hat{a}z_1=z_2\check{a}z_1$ hold for every $a\in S$; that is, if the following condition, which we shall call \textbf{multiple associativity} or \textbf{multi-associativity}, holds.$$(az_1b)z_2c=az_1(bz_2c)\text{ and }(az_2b)z_1c=az_2(bz_1c),\,\forall a,b,c\in S.$$}
\end{definition}The compatibility is a symmetric relation; however, it is neither reflexive nor transitive unless one restricts the domain. Indeed, $S^{S\times S}_{sg}$ can be defined as the set of those elements in $S^{S\times S}$ each of which is compatible with itself; thus the compatibility is reflexive on $S^{S\times S}_{sg}$. Also note that the compatibility is transitive on $S^{S\times S}_{nd}$ which follows from Lemma~\ref{lm:key}.

Whenever $z_1$ and $z_2$ are compatible, we simply write $az_1bz_2c$ for $(az_1b)z_2c$ or $az_1(bz_2c)$ without ambiguity, and write $z_1\hat{a}z_2$ for $z_1\check{a}z_2$, where $z_1,z_2\in S^{S\times S}$, $a,b,c\in S$. In this case $z_1\hat{a}z_2$ and $z_1\hat{b}z_2$ are compatible for all $a,b\in S$, for $(c(z_1\hat{a}z_2)d)(z_1\hat{b}z_2)f=((cz_1a)z_2d)z_1(bz_2f)=(cz_1a)z_2(dz_1(bz_2f))=c(z_1\hat{a}z_2)(d(z_1\hat{b}z_2)f),\forall a,b,c,d,f\in S$. In particular, $z_1\hat{a}z_2$ is compatible with itself; thus $z_1\hat{a}z_2\in S^{S\times S}_{sg}$ even if $z_1$ or $z_2$ is not compatible with itself.

Now we collect all the elements that are compatible with a given element $z\in S^{S\times S}$.
\begin{definition}\label{de:dual}{\em For each $z\in S^{S\times S}$, define$$S^*_z:=\{z'\in S^{S\times S}\,|\,(azb)z'c=az(bz'c)\text{ and }(az'b)zc=az'(bzc),\,\forall a,b,c\in S\};$$that is, $S^*_z$ is the set of all elements in $S^{S\times S}$ that is compatible with $z$. We call $S^*_z$ the \textbf{dual} of $S$ with respect to the binary operation $z$.}
\end{definition}The reason why we call $S^*_z$ the dual of $S$ will become clear in Theorem~\ref{th:main}~(1). Note that $z\in S^*_z$ if and only if $z\in S^{S\times S}_{sg}$. Also note that for $z_1,z_2\in S^{S\times S}$, $z_1\in S^*_{z_2}$ if and only if $z_2\in S^*_{z_1}$ since compatibility is symmetric. However, $z_1\in S^*_{z_2}$ does not always imply that $S^*_{z_1}=S^*_{z_2}$ unless $z_1\in S^{S\times S}_{gr}$, or $z_2\in S^{S\times S}_{gr}$, or $z_1,z_2\in S^{S\times S}_{nd}$. This means that $z_1$ may not be compatible with all elements of $S^*_{z_2}$ even though it is compatible with $z_2$. See Lemma~\ref{lm:key}, Corollary~\ref{co:main}, and Example~\ref{ex:n=3} (toward the end of the example).

The main result of this paper is Theorem~\ref{th:main}. In part~(1) of the theorem we show that if $z\in S^{S\times S}_{mn}$, then there is a bijection $\phi$ from $S$ onto $S^*_z$ such that the monoid $(S,\phi(a))$ is isomorphic to $(S^*_z,\hat{a})$ via $\phi$ for each $a\in S$. Part~(2) of the theorem shows that $S^{S\times S}_{gr}$ is partitioned into copies of $S^*_z$ if $z\in S^{S\times S}_{gr}$. Corollary~\ref{co:main} provides a new characterization of group binary operations.

This work was motivated by the author's precedent works \cite{Kaneda 2003}, \cite{Kaneda 2007}, and \cite{Kaneda and Paulsen 2004} (joint with V.~I.~Paulsen) in which operator algebra products are characterized using quasi-multipliers of operator spaces. That is, the operator algebra products a given operator space can be equipped with are precisely the bilinear mappings on the operator space that are implemented by contractive quasi-multipliers. The present paper is the outcome of an attempt to introduce a counterpart to quasi-multipliers in the most primitive setting in pure algebra. The elements of $S^{S\times S}$ play roles more or less similar to the quasi-multipliers.
\section{Results.}\label{section:results}
The following lemma is useful throughout this section.
\begin{lemma}\label{lm:key}Let $z_1\in S^{S\times S}_{nd}$ and $z_2\in S^{S\times S}$. If $z_1\in S^*_{z_2}$, or, equivalently, $z_2\in S^*_{z_1}$, then $S^*_{z_1}\subseteq S^*_{z_2}$. Therefore, if, in addition, $z_2\in S^{S\times S}_{nd}$, then $S^*_{z_1}=S^*_{z_2}$.
\end{lemma}
\begin{proof}Suppose that $z_1\in S^{S\times S}_{nd}$, $z_2\in S^{S\times S}$, and $z_1\in S^*_{z_2}$. Let $z\in S^*_{z_1}$ and $a,b,c\in S$. Then there exist $d,f\in S$ such that $b=dz_1f$, so that
$(azb)z_2c=(az(dz_1f))z_2c=((azd)z_1f)z_2c=(azd)z_1(fz_2c)=az(dz_1(fz_2c))=az((dz_1f)z_2c)=az(bz_2c)$. Similarly, $(az_2b)zc=az_2(bzc)$, hence $S^*_{z_1}\subseteq S^*_{z_2}$.
\end{proof}It is interesting to note in the proof above that $z_1$ plays the role of a ``catalyst'' to make $z$ and $z_2$ compatible.

The following proposition tells us that if $z\in S^{S\times S}_{nd}$, then any two elements in $S^*_z$ are compatible (we shall refer to this property as the \emph{pairwise compatibility} of the elements of $S^*_z$), and $S^*_z$ is closed under the binary operation $\hat{a}$ for each $a\in S$, and the multiple associativity works in $S^*_z$. That is, $(S^*_z,\hat{a})$ is a semigroup for each $a\in S$. Note, however, that $z\notin S^*_z$ in general as remarked after Definition~\ref{de:dual}. Also note that assuming $z\in S^{S\times S}_{nd}$ is essential in the proposition since otherwise the multiple associativity in $S$ in part~(1) might not hold. See the end of Example~\ref{ex:n=3}.
\begin{proposition}\label{pr:dual}Let $z\in S^{S\times S}_{nd}$. Then the following hold.
\begin{enumerate}
  \item\emph{(Multiple Associativity in $S$)} $(az_1b)z_2c=az_1(bz_2c),\forall a,b,c\in S,\forall z_1,z_2\in S^*_z$. In particular, $S^*_z\subseteq S^{S\times S}_{sg}$. Furthermore, if $z\notin S^{S\times S}_{sg}$, then $S^*_z\subseteq S^{S\times S}_{sg}\setminus S^{S\times S}_{nd}$.
  \item\emph{(Closedness of $S^*_z$)} If $z_1,z_2\in S^*_z$, then $z_1\hat{a}z_2\in S^*_z,\forall a\in S$.
  \item\emph{(Multiple Associativity in $S^*_z$)} $(z_1\hat{a}z_2)\hat{b}z_3=z_1\hat{a}(z_2\hat{b}z_3),\forall z_1,z_2,z_3\in S^*_z,\forall a,b\in S$.
\end{enumerate}
\end{proposition}
\begin{proof}
\begin{enumerate}
  \item Let $z_1,z_2\in S^*_z$. Then by Lemma~\ref{lm:key} $S^*_z\subseteq S^*_{z_1}$ since $z\in S^{S\times S}_{nd}$. But $z_2\in S^*_z$, so that $z_2\in S^*_{z_1}$; that is, $z_1$ is compatible with $z_2$. To see the last claim, assume that $z\notin S^{S\times S}_{sg}$ and $S^*_z\cap S^{S\times S}_{nd}\ne\emptyset$. Pick $z_3\in S^*_z\cap S^{S\times S}_{nd}$, then by Lemma~\ref{lm:key} $S^*_{z_3}=S^*_z$, so that $z\in S^*_{z_3}=S^*_z\subseteq S^{S\times S}_{sg}$, a contradiction.
  \item For all $a,b,c,d\in S$, $(bzc)(z_1\hat{a}z_2)d=((bzc)z_1a)z_2d=(bz(cz_1a))z_2d=bz((cz_1a)z_2d)=bz(c(z_1\hat{a}z_2)d)$. Similarly, $(b(z_1\hat{a}z_2)c)zd=b(z_1\hat{a}z_2)(czd)$.
  \item Let $a,b,c,d\in S$ and $z_1,z_2,z_3\in S^*_z$. Then repeated use of the pairwise compatibility of elements in $S^*_z$ proved in part~(1) yields that $c((z_1\hat{a}z_2)\hat{b}z_3)d=(c(z_1\hat{a}z_2)b)z_3d=((cz_1a)z_2b)z_3d=(cz_1(az_2b))z_3d=cz_1((az_2b)z_3d)=cz_1(a(z_2\hat{b}z_3)d)=c(z_1\hat{a}(z_2\hat{b}z_3))d$, where in the last equality we used the fact that $z_2\hat{b}z_3\in S^*_z$, a conclusion from part~(2).
\end{enumerate}
\end{proof}With reference to part~(1) of the above proposition, we leave the following as open question.
\begin{question*}{\em Does $S^*_z$ exhaust $S^{S\times S}_{sg}\setminus S^{S\times S}_{nd}$ when $z$ moves around in $S^{S\times S}_{nd}\setminus S^{S\times S}_{sg}$; that is, $\bigcup\{S^*_z\,|\,z\in S^{S\times S}_{nd}\setminus S^{S\times S}_{sg}\}=S^{S\times S}_{sg}\setminus S^{S\times S}_{nd}$? Or, at least $S^*_z\cap S^{S\times S}_{nd}\ne\emptyset$ whenever $z\in S^{S\times S}_{sg}$? }
\end{question*}An affirmative answer to the first part implies one to the second which says that every associative binary operation is compatible with at least one nondegenerate binary operation (which may or may not be associative).

Now we are in a position to state our main result.
\begin{theorem}\label{th:main}Suppose that $S\neq\emptyset$.
\begin{enumerate}
  \item\emph{(Duality)} Let $z_0\in S^{S\times S}_{mn}$, and let $e\in S$ be the identity of the monoid $(S,z_0)$. Then there is a bijection $\phi$ from $S$ onto $S^*_{z_0}$ such that $\phi(e)=z_0$, and for each $a\in S$, the semigroup $(S,\phi(a))$ is isomorphic to the semigroup $(S^*_{z_0},\hat{a})$ via $\phi$. In particular, $(S^*_{z_0},\hat{e})$ is a monoid with identity $z_0$.
  \item Let $z_0\in S^{S\times S}_{gr}$, and let $e\in S$ be the identity of the group $(S,z_0)$. Then, in addition to the conclusions of (1), the following hold.
  \begin{enumerate}
    \item$S^*_{z_0}\subseteq S^{S\times S}_{gr}$, and for every $z\in S^*_{z_0}$, $S^*_z=S^*_{z_0}$, and the group $(S,z)$ is isomorphic to $(S,z_0)$ (hence by (1) it is also isomorphic to $(S^*_{z_0},\hat{a})$ for every $a\in S$).
    \item Let $S^{S\times S}_{gr}(e)$ be the set of those elements $z\in S^{S\times S}_{gr}$ for which $e$ is the identity of the group $(S,z)$. Then $S^{S\times S}_{gr}$ is partitioned into $S^{S\times S}_{gr}=\coprod\{S^*_{z}\,|\,z\in S^{S\times S}_{gr}(e)\}$ (disjoint union).
    \item Let us say that $z_1\in S^{S\times S}_{gr}(e)$ and $z_2\in S^{S\times S}_{gr}(e)$ are \emph{\textbf{equivalent}} and write $z_1\sim z_2$ if $(S,z_1)$ and $(S,z_2)$ are isomorphic (obviously $\sim$ is an equivalence relation). Let $R:=\{z_{\lambda}\,|\,\lambda\in\Lambda\}$ be a complete set of representatives of the equivalence classes in $S^{S\times S}_{gr}(e)/{\sim}$, where $\Lambda$ is an index set, and denote the equivalence class of $z_{\lambda}$ by $[z_{\lambda}]$. Then each element of $R$ equips $S$ with a distinct group structure on $S$, and the elements of $R$ exhaust all the possible group structures on $S$. In particular, if $S$ is a finite set, then the number of distinct group structures on $S$ is $|R|(=|\Lambda|)$. Furthermore, let us denote by $\operatorname{Sym}_e(S)$ the group of permutations on $S$ that fix $e$, and for each $z_{\lambda}\in R$, let us say that $\sigma_1,\in\operatorname{Sym}_e(S)$ and $\sigma_2,\in\operatorname{Sym}_e(S)$ are \emph{\textbf{$z_{\lambda}$-equivalent}} and write $\sigma_1\sim_{z_{\lambda}}\sigma_2$ if $\sigma_2^{-1}\circ\sigma_1\in\operatorname{Aut}(S,z_{\lambda})$, where $\operatorname{Aut}(S,z_{\lambda})$ is the group of automorphisms on $(S,z_{\lambda})$ which is a subgroup of $\operatorname{Sym}_e(S)$. Then $[z_{\lambda}]$ has the same cardinality as the quotient $\operatorname{Sym}_e(S)/\operatorname{Aut}(S,z_{\lambda})$. In particular, if $|S|=n$, a positive integer, then $|[z_{\lambda}]|=|S_{n-1}|/|\operatorname{Aut}(S,z_{\lambda})|=(n-1)!/|\operatorname{Aut}(S,z_{\lambda})|$, where $S_{n-1}$ is the symmetric group of degree $n-1$.
  \end{enumerate}
\end{enumerate}
\end{theorem}
\begin{proof}
\begin{enumerate}
  \item Define $\phi:S\to S^*_{z_0}$ by $\phi(a):=z_0\hat{a}z_0,\forall a\in S$. Clearly the range is in $S^*_{z_0}$ noting that $z_0$ is compatible with itself since $z_0\in S^{S\times S}_{mn}\subseteq S^{S\times S}_{sg}$. We shall show that this $\phi$ has the desired properties. It is obvious that $z_0=z_0\hat{e}z_0$, so that $\phi(e)=z_0$. It is also easy to see that $\phi$ is one-to-one, for $\phi(a)=\phi(b)\,(a,b\in S)$ yields that $a=e(z_0\hat{a}z_0)e=e\phi(a)e=e\phi(b)e=e(z_0\hat{b}z_0)e=b$. To see that $\phi$ is onto, first note that for every $z\in S^*_{z_0}$ and every $a,b\in S$, we have that $azb=az(ez_0b)=(aze)z_0b=((az_0e)ze)z_0b=(az_0(eze))z_0b$, which implies that the value of $eze$ completely determines $z$ as an element of $S^*_{z_0}$. When $a$ takes all elements of $S$, $e\phi(a)e$ takes all elements of $S$ since $e\phi(a)e=e(z_0\hat{a}z_0)e=a,\forall a\in S$. Thus $\phi$ is onto. Finally, the assertion that $\phi:(S,\phi(a))\to(S^*_{z_0},\hat{a})$ is an isomorphism follows from $d\phi(b\phi(a)c)f=(dz_0((bz_0a)z_0c))z_0f=((dz_0(bz_0a))z_0c)z_0f=(d\phi(b)a)\phi(c)f=d(\phi(b)\hat{a}\phi(c))f,\,\forall a,b,c,d,f\in S$.
  \item
  \begin{enumerate}
    \item Let $\phi$ be as in (1) and $z\in S^*_{z_0}$. Then by (1) there exists an $a\in S$ such that $z=\phi(a)$. We denote the inverse of each element $b\in S$ in the group $(S,z_0)$ by $b^{-1}$. Define $\psi$ from the group $(S,z_0)$ to the semigroup $(S,z)$ by $\psi(b):=bz_0a^{-1},\forall b\in S$. It is straightforward to check that $\psi$ is a homomorphism noting that $z=z_0az_0$, so that $(S,z)$ is a group and $S^*_{z_0}\subseteq S^{S\times S}_{gr}$, hence by Lemma~\ref{lm:key} $S^*_{z_0}=S^*_z$. It is also an easy routine work to check that $\psi$ is one-to-one and onto; thus $(S,z_0)$ is isomorphic to $(S,z)$. (In fact, $a^{-1}$ is the identity of $(S,z)$, and $a^{-1}z_0b^{-1}z_0a^{-1}$ is the inverse of $b\in S$ in $(S,z)$.)
    \item Suppose that $z_1,z_2\in S^{S\times S}_{gr}(e)$ and $S^*_{z_1}\cap S^*_{z_2}\ne\emptyset$. Pick $z_3\in S^*_{z_1}\cap S^*_{z_2}$ the right-hand side of which is a subset of $S^{S\times S}_{gr}$ by part~(a), then by Lemma~\ref{lm:key} $S^*_{z_1}=S^*_{z_3}=S^*_{z_2}$; that is, $z_1$ and $z_2$ are compatible. Therefore for all $a,b\in S$, $az_1b=az_1(ez_2b)=(az_1e)z_2b=az_2b$, hence $z_1=z_2$. Next let $z\in S^{S\times S}_{gr}$, and let $a$ be the identity of the group $(S,z)$, and let $e^{-1}$ be the inverse of $e$ in $(S,z)$. Define $z_4\in S^{S\times S}_{sg}$ by $z_4:=z\widehat{e^{-1}}z$. Clearly, $z_4\in S^*_z\subseteq S^{S\times S}_{gr}$, hence $(S,z_4)$ is a group and $z\in S^*_{z_4}$ as well. Since $ez_4b=eze^{-1}zb=azb=b$ and similarly $bz_4e=b$, $\forall b\in S$, we know that $e$ is the identity of $(S,z_4)$, and hence $z_4\in S^{S\times S}_{gr}(e)$.
    \item We shall show the only nontrivial statement that for each $z_{\lambda}\in R$, $[z_{\lambda}]$ has the same cardinality as $\operatorname{Sym}_e(S)/\operatorname{Aut}(S,z_{\lambda})$. For each $\sigma\in\operatorname{Sym}_e(S)$ define $z^{\sigma}\in S^{S\times S}$ by $\sigma(a)z^{\sigma}\sigma(b)=\sigma(az_{\lambda}b),\forall a,b\in S$. Note that this $z^{\sigma}$ is the only binary operation on $S$ that makes $\sigma:(S,z_{\lambda})\to(S,z^{\sigma})$ an isomorphism. Then $(S,z^{\sigma})$ is a group with identity $\sigma(e)=e$ which is isomorphic to $(S,z_{\lambda})$, and thus $z^{\sigma}\in[z_{\lambda}]$. Define a function $\varphi:\operatorname{Sym}_e(S)\to[z_{\lambda}]$ by $\varphi(\sigma)=z^{\sigma}$. It is easy to see that $\varphi$ is onto. Indeed, let $z\in[z_{\lambda}]$. Then there is an isomorphism $\sigma_0\in\operatorname{Sym}_e(S)$ from $(S,z_{\lambda})$ onto $(S,z)$; that is, $\sigma_0(az_{\lambda}b)=\sigma_0(a)z\sigma_0(b),\forall a,b\in S$. Thus $z=z^{\sigma_0}$, and hence $\varphi$ is onto. Now suppose that $\sigma_1,\sigma_2\in\operatorname{Sym}_e(S)$ and $z^{\sigma_1}=z^{\sigma_2}$. Then $\sigma_1(az_{\lambda}b)=\sigma_1(a)z^{\sigma_1}\sigma_1(b)=\sigma_1(a)z^{\sigma_2}\sigma_1(b)=\sigma_2(\sigma_2^{-1}(\sigma_1(a)))z^{\sigma_2}\sigma_2(\sigma_2^{-1}(\sigma_1(b)))=\sigma_2(\sigma_2^{-1}(\sigma_1(a))z_{\lambda}\sigma_2^{-1}(\sigma_1(b)))$, so that $\sigma_2^{-1}(\sigma_1(az_{\lambda}b))=\sigma_2^{-1}(\sigma_1(a))z_{\lambda}\sigma_2^{-1}(\sigma_1(b)),\linebreak\forall a,b\in S$, which tell us that $\sigma_2^{-1}\circ\sigma_1\in\operatorname{Aut}(S,z_{\lambda})$; that is, $\sigma_1\sim_{z_{\lambda}}\sigma_2$.
  \end{enumerate}
\end{enumerate}
\end{proof}
\begin{remark}\label{rm:main}{\em
\begin{enumerate}
  \item The key part in the proof of part~(1) of the theorem is that the value of $eze$ completely determines the structure of the semigroup $(S,z)$.
  \item In part~(1) of the theorem $S^*_{z_0}$ need not be a subset of $S^{S\times S}_{nd}$ in general even if $z_0\in S^{S\times S}_{mn}$ (see Example~\ref{ex:n=3}). Also it could happen that $\emptyset\ne S^*_{z_1}\cap S^*_{z_2}\subset S^{S\times S}_{sg}\setminus S^{S\times S}_{nd}$ for $z_1,z_2\in S^{S\times S}_{mn}$ (see also Example~\ref{ex:n=3}). For these reasons $S^{S\times S}_{mn}$ cannot be partitioned in general in the way we did for $S^{S\times S}_{gr}$ in part~(2a) of the theorem.
  \item If we express the binary operation $z_0\in S^{S\times S}_{mn}$ by concatenation and the new product $\phi(a)$ by ``$\,\cdot\,$'' in the proof of (1), then $b\cdot c=bac$ for $b,c\in S$. This means that each binary operation in $S^*_{z_0}$ can be obtained by ``sandwiching'' each element of $S$. In particular, if $z_0\in S^{S\times S}_{gr}$, then part~(2a) of the theorem is saying that every element $z\in S_{z_0}$ is ``equivalent'' and every element is qualified for an identity by redefining a binary operation by sandwiching.
  \item In (1) (respectively, (2)) of the theorem, if $S$ is equipped with a topology $\tau$ (the set of open sets in $S$), then $\tau^*:=\{\{z_0bz_0\,|\,b\in U\}\,|\,U\in\tau\}$ defines a topology on $S^*_{z_0}$, and with this topology $(S,\phi(a))$ is isomorphic to $(S^*_{z_0},\hat{a})$ as topological monoids (respectively, topological groups) via $\phi$ for each $a\in S$; that is, $\phi$ is a homeomorphism as well as an isomorphism.
  \item Given a binary operation $z\in S^{S\times S}$ and a permutation $\sigma$ on $S$, the binary operation $z^{\sigma}$ defined by $\sigma(a)z^{\sigma}\sigma(b)=\sigma(azb),\forall a,b\in S$ is the only one that makes $\sigma$ an isomorphism from the magma $(S,z)$ onto the magma $(S,z^{\sigma})$. Of course, $z^{\sigma}=z$ if and only if $\sigma\in\operatorname{Aut}(S,z)$.
  \item In (2c) of the theorem $\operatorname{Aut}(S,z_{\lambda})$ is not a normal subgroup of $\operatorname{Sym}_e(S)$ in general. See Example~\ref{ex:n=4}.
\end{enumerate}}
\end{remark}The following corollary tells us that no element $z\notin S^{S\times S}_{gr}$ is compatible with any element of $S^{S\times S}_{gr}$. This property distinguishes the elements of $S^{S\times S}_{gr}$ from the other binary operations, and provides a new characterization of group binary operations.
\begin{corollary}\label{co:main}Let $z_1\in S^{S\times S}_{gr}$ and $z_2\in S^{S\times S}$. Then either $S^*_{z_1}\cap S^*_{z_2}=\emptyset$ or $S^*_{z_1}=S^*_{z_2}$ occurs.
\end{corollary}
\begin{proof}Suppose that $S^*_{z_1}\cap S^*_{z_2}\ne\emptyset$ and pick $z_3\in S^*_{z_1}\cap S^*_{z_2}$. By Theorem~\ref{th:main}~(2a), $S^*_{z_1}\subseteq S^{S\times S}_{gr}$, so that $z_3\in S^{S\times S}_{gr}$, hence $S^*_{z_1}=S^*_{z_3}\subseteq S^*_{z_2}$ by Lemma~\ref{lm:key}. But $z_3\in S^*_{z_2}$ implies that $z_2\in S^*_{z_3}=S^*_{z_1}\subseteq S^{S\times S}_{gr}$, so that $S^*_{z_2}\subseteq S^*_{z_1}$ by Lemma~\ref{lm:key} again.
\end{proof}
\section{Examples}\label{section:examples}
Although our results in Section~\ref{section:results} apply to sets of any cardinalities, in this section we restrict ourselves to two simple examples of finite sets to illustrate situations in Theorem~\ref{th:main} and clarify remarks made in Sections~\ref{section:intro}~and~\ref{section:results}. These examples provide only commutative monoids or groups, but the results in Section~\ref{section:results} are valid for noncommutative cases as well.
\begin{example}\label{ex:n=3}{\em Let $S=\{a,b,c\}$, where $a$, $b$, and $c$ are distinct. Although $|S^{S\times S}|=19683$, $S^{S\times S}_{gr}$ consists of only $3$ elements, say $z_1$, $z_2$, and $z_3$, which are defined by the following Cayley tables.
$$\begin{array}{c|ccc}
z_1 & a & b & c\\
\hline
a & a & b & c\\
b & b & c & a\\
c & c & a & b
\end{array}\quad
\begin{array}{c|ccc}
z_2 & a & b & c\\
\hline
a & c & a & b\\
b & a & b & c\\
c & b & c & a
\end{array}\quad
\begin{array}{c|ccc}
z_3 & a & b & c\\
\hline
a & b & c & a\\
b & c & a & b\\
c & a & b & c
\end{array}$$It is easy to verify that $S^*_{z_1}=S^*_{z_2}=S^*_{z_3}=\{z_1,z_2,z_3\}=S^{S\times S}_{gr}$. The elements of $S$ define binary operations on $S^*_{z_1}(=S^*_{z_2}=S^*_{z_3})$ as follows.
$$\begin{array}{c|ccc}
\hat{a} & z_1 & z_2 & z_3\\
\hline
z_1 & z_1 & z_2 & z_3\\
z_2 & z_2 & z_3 & z_1\\
z_3 & z_3 & z_1 & z_2
\end{array}\quad
\begin{array}{c|ccc}
\hat{b} & z_1 & z_2 & z_3\\
\hline
z_1 & z_3 & z_1 & z_2\\
z_2 & z_1 & z_2 & z_3\\
z_3 & z_2 & z_3 & z_1
\end{array}\quad
\begin{array}{c|ccc}
\hat{c} & z_1 & z_2 & z_3\\
\hline
z_1 & z_2 & z_3 & z_1\\
z_2 & z_3 & z_1 & z_2\\
z_3 & z_1 & z_2 & z_3
\end{array}$$We see that $(S,z_1)\cong(S,z_2)\cong(S,z_3)\cong(S^*_{z_1},\hat{a})\cong(S^*_{z_1},\hat{b})\cong(S^*_{z_1},\hat{c})$ as parts~(1)~and~(2a) of Theorem~\ref{th:main} assert. In this case $S^{S\times S}_{gr}$ is partitioned into only one component which is itself, and $[z_1]=1$, which concludes the trivial fact that there is only one group structure of order $3$; that is the cyclic group of order $3$. In this example $\operatorname{Sym}_a(S)\cong S_2$ and it is known that the automorphism group on a cyclic group of order $3$ is a cyclic group of order $2$, so that $|\operatorname{Sym}_a(S)|/|\operatorname{Aut}(S,z_1)|=|S_2|/2=2!/2=1$ which is equal to $[z_1]$ as asserted in Theorem~\ref{th:main}~(2c).

Now let $z_4\in S^{S\times S}_{mn}\setminus S^{S\times S}_{gr}$ be defined in the left table below. Then it is easy to know that $S^*_{z_4}=\{z_4,z_5,z_6\}$ referring to Remark~\ref{rm:main}~(1) if necessary, where $z_5$ and $z_6$ are also defined below.
$$\begin{array}{c|ccc}
z_4 & a & b & c\\
\hline
a & a & b & c\\
b & b & a & c\\
c & c & c & c
\end{array}\quad
\begin{array}{c|ccc}
z_5 & a & b & c\\
\hline
a & b & a & c\\
b & a & b & c\\
c & c & c & c
\end{array}\quad
\begin{array}{c|ccc}
z_6 & a & b & c\\
\hline
a & c & c & c\\
b & c & c & c\\
c & c & c & c
\end{array}$$Although $z_4,z_5\in S^{S\times S}_{mn}$, $z_6\notin S^{S\times S}_{nd}$, which provides an example of the first statement of Remark~\ref{rm:main}~(2). The elements of $S$ define binary operations on $S^*_{z_4}$ as follows.
$$\begin{array}{c|ccc}
\hat{a} & z_4 & z_5 & z_6\\
\hline
z_4 & z_4 & z_5 & z_6\\
z_5 & z_5 & z_4 & z_6\\
z_6 & z_6 & z_6 & z_6
\end{array}\quad
\begin{array}{c|ccc}
\hat{b} & z_4 & z_5 & z_6\\
\hline
z_4 & z_5 & z_4 & z_6\\
z_5 & z_4 & z_5 & z_6\\
z_6 & z_6 & z_6 & z_6
\end{array}\quad
\begin{array}{c|ccc}
\hat{c} & z_4 & z_5 & z_6\\
\hline
z_4 & z_6 & z_6 & z_6\\
z_5 & z_6 & z_6 & z_6\\
z_6 & z_6 & z_6 & z_6
\end{array}$$We see that $(S,z_4)\cong(S^*_{z_4},\hat{a})$, $(S,z_5)\cong(S^*_{z_4},\hat{b})$, and $(S,z_6)\cong(S^*_{z_4},\hat{c})$ as Theorem~\ref{th:main}~(1) asserts.

Next we define $z_7\in S^{S\times S}_{mn}\setminus S^{S\times S}_{gr}$ by the left table below. Then we see that $S^*_{z_7}=\{z_7,z_8,z_6\}$, where $z_8$ is defined by the right table below.
$$\begin{array}{c|ccc}
z_7 & a & b & c\\
\hline
a & a & b & c\\
b & b & b & c\\
c & c & c & c
\end{array}\quad
\begin{array}{c|ccc}
z_8 & a & b & c\\
\hline
a & b & b & c\\
b & b & b & c\\
c & c & c & c
\end{array}$$Two distinct sets $S^*_{z_4}$ and $S^*_{z_7}$ with $z_4,z_7\in S^{S\times S}_{mn}$ share a common element $z_6\in S^{S\times S}_{sg}\setminus S^{S\times S}_{nd}$, which provides an example of the second statement of Remark~\ref{rm:main}~(2). It is not hard to see that $S^*_{z_6}$ consists of all elements $z$ of the form
$$\begin{array}{c|ccc}
z & a & b & c\\
\hline
a & * & * & c\\
b & * & * & c\\
c & c & c & c
\end{array}$$where each $*$ represents any element of $S$ and different $*$'s can take different elements of $S$. Thus $|S^*_{z_6}|=3^4=81$, and $S^*_{z_4}\subsetneq S^*_{z_6}$ with $z_4\in S^*_{z_6}$, which provides an example of a remark in the paragraph after Definition~\ref{de:dual}. Also note that $z_4$ and $z_7$ are not compatible while $z_4,z_7\in S^*_{z_6}$, which provides an example of the remark right before Proposition~\ref{pr:dual}.}
\end{example}
\begin{example}\label{ex:n=4}{\em Let $S:=\{a,b,c,d\}$, where $a$, $b$, $c$, and $d$ are distinct. Let $z_1,z_2,z_3,z_4\in S^{S\times S}_{gr}$ be defined as follows.
$$\begin{array}{c|cccc}
z_1 & a & b & c &d\\
\hline
a & a & b & c & d\\
b & b & c & d & a\\
c & c & d & a & b\\
d & d & a & b & c
\end{array}\quad
\begin{array}{c|cccc}
z_2 & a & b & c & d\\
\hline
a & a & b & c & d\\
b & b & a & d & c\\
c & c & d & b & a\\
d & d & c & a & b
\end{array}\quad
\begin{array}{c|cccc}
z_3 & a & b & c & d\\
\hline
a & a & b & c & d\\
b & b & d & a & c\\
c & c & a & d & b\\
d & d & c & b & a
\end{array}
\quad
\begin{array}{c|cccc}
z_4 & a & b & c & d\\
\hline
a & a & b & c & d\\
b & b & a & d & c\\
c & c & d & a & b\\
d & d & c & b & a
\end{array}$$It is easy to verify that the above are all possible groups of which $a$ is the identity, hence by Theorem~\ref{th:main}~(2b) we can conclude that $S^{S\times S}_{gr}=S^*_{z_1}\amalg S^*_{z_2}\amalg S^*_{z_3}\amalg S^*_{z_4}$ (disjoint union). As one can easily verify, $(S,z_1)\cong(S,z_2)\cong(S,z_3)$ are cyclic groups of order~$4$, and $(S,z_4)$ is a Klein four-group, so we reconfirm the well-known fact that there are only two distinct group structures of order~$4$, and we have that $|[z_1]|=3$ and $|[z_4]|=1$. In this example $\operatorname{Sym}_a(S)\cong S_3$. Since it is known that the automorphism group on a cyclic group of order~$4$ is a cyclic group of order~$2$, $|\operatorname{Sym}_a(S)|/|\operatorname{Aut}(S,z_1)|=|S_3|/2=3!/2=3$ which is equal to $|[z_1]|$. It is also known that the automorphism group on a Klein four-group is isomorphic to the symmetric group $S_3$, so that $|\operatorname{Sym}_a(S)|/|\operatorname{Aut}(S,z_4)|=|S_3|/|S_3|=1$ which is equal to $|[z_4]|$.

Now let $\sigma_0$ be the permutation on $S$ interchanging $b$ and $c$, and let $\tau_0$ be the permutation on $S$ interchanging $b$ and $d$. Then $\sigma_0\in\operatorname{Sym}_e(S)$, and it is easy to check that $\tau_0\in\operatorname{Aut}(S,z_1)$ and that $z^{\tau_0\circ\sigma_0}=z_3$. (For the superscript notation, see the proof of Theorem~\ref{th:main}~(2c).) However, for any $\tau\in\operatorname{Aut}(S,z_1)$, $z^{\sigma_0\circ\tau}=z_2$, so $z^{\tau_0\circ\sigma_0}\ne z^{\sigma_0\circ\tau}$. Since $z^{\sigma}$ is the only bilinear operation that makes $\sigma:(S,z_1)\to(S,z^{\sigma})$ isomorphism (see also the proof of Theorem~\ref{th:main}~(2c)), $\sigma_0\circ\tau\ne\tau_0\circ\sigma_0$. Thus $\sigma_0\operatorname{Aut}(S,z_1)\ne\operatorname{Aut}(S,z_1)\sigma_0$, and hence $\operatorname{Aut}(S,z_1)$ is not a normal subgroup of $\sigma_0\in\operatorname{Sym}_a(S)$ as noted in Remark~\ref{rm:main}~(6).}
\end{example}

\end{document}